\documentclass[final,leqno,onefignum,onetabnum,oneeqnum,]{siamltex1213}

\usepackage{amsfonts}
\usepackage{amssymb}
\usepackage{amscd}
\usepackage{amsmath}
\usepackage{hyperref}
\usepackage[english]{babel}

\newtheorem{example}{Example}[section]

\title{Projection process with definable right-hand side\thanks{This work was partially supported by
the Sofia University "St. Kliment Ohridski" fund "Research \& Development"
under contract 08/26.03.2015} and by the Bulgarian National Scientific Fund under Grant DFNI-I02/10.} 

\author{Mira Bivas\footnotemark[2]\ \footnotemark[3]
\and Nadezhda Ribarska\footnotemark[2]\ \footnotemark[4]\ \footnotemark[5]
}

\begin{document}
\maketitle
\slugger{sicon}{xxxx}{xx}{x}{x--x}

\renewcommand{\thefootnote}{\fnsymbol{footnote}}

\footnotetext[2]{Faculty of Mathematics and Informatics,  Sofia University, James
Bourchier Boul. 5,  1126 Sofia, Bulgaria}
\footnotetext[3]{email: mira.bivas@gmail.com}
\footnotetext[4]{Institute of
Mathematics and Informatics, BAS, G.Bonchev str. 8, 1113 Sofia,Bulgaria}
\footnotetext[5]{email: ribarska@fmi.uni-sofia.bg}

\renewcommand{\thefootnote}{\arabic{footnote}}

\begin{abstract}
We study the relation between sweeping processes with the cone of limiting normals and projection processes. We prove the existence of solution of a perturbed sweeping process with the cone of limiting normals and of nonstationary projection process, provided the sets involved are definable (and are moving in a definable way) in some o-minimal structure. An application to a crowd motion model is presented.
\end{abstract}



\pagestyle{myheadings}
\thispagestyle{plain}
\markboth{Mira Bivas and Nadezhda Ribarska}{Projection process with definable right-hand side}

\section{Introduction}
The classical sweeping process has been introduced and thoroughly
studied in the 70s by J.J. Moreau  (cf., e.g. \cite{M1}). General
motivation arising from Mechanics appeared in \cite{M2} and extensive mechanical models can also be found in \cite{M3}.
The mathematical formulation of a sweeping process  is  the following
constrained differential inclusion
\begin{equation}\label{sweep}
\begin{array}{l} \dot x(t) \in -N_{C(t)} (x(t))\\ x(0) = x_0
\in C(0)\\ x(t) \in C(t),\end{array}
\end{equation}
where $C(t)$ is a given moving closed set and $N_{C(t)} (x(t))$ is
the normal cone (in some sense) to $C(t)$ at $x(t)$. In the papers by Moreau mentioned above the sets $C(t)$ are convex subsets of a Hilbert space and they are moving in an absolutely continuous way.

Two years later the first paper by Moreau the study of resource
allocation mechanisms led C.Henry (c.f. \cite{H1}, \cite{H}) to the differential inclusion
\begin{equation}\label{projection}
\begin{array}{l} \dot x(t) \in Proj_{T_C (x(t))} F(x(t))\\ x(0) = x_0
\in C\\ x(t) \in C ,\end{array}
\end{equation}
where $C$ is a closed convex set in a finite dimensional space, $T_C(x(t))$ is the tangent cone to $C$ and $F$ is an upper semicontinuous multivalued mapping with nonempty compact convex values. Soon the close relation between these two problems has been recognised. B.Cornet (\cite{Cornet}) proved the equivalence of (\ref{projection}) and the sweeping process (\ref{sweep}) with the perturbation $F$ added to the right-hand side and constant set $C$ which is assumed to be Clarke regular.

The investigation of the sweeping process has been carried out under
different assumptions on the phase space, on the geometry of the moving set, on the way the set is moving, on the possible perturbations etc. Nowadays there exists an extensive literature
on the subject. The reader is referred to \cite{Thib} and the references therein for a detailed discussion. The sweeping process with perturbation $F$ added to the right-hand side for prox-regular sets in a Hilbert space is studied in \cite{EdmThib1}, \cite{EdmThib2}. The connection of this problem (with stationary prox-regular set) to the projection process (\ref{projection}) is done in \cite{Serea}. The differential inclusion \eqref{projection} appeared again in a crowd motion model (\cite{BerVen}, \cite{BerVen2}) and has been shown again to be equivalent to a sweeping process in the case of lack of obstacles due to the prox-regularity of the sets involved.

We are interested in investigating these problems when the geometry of the sets is not regular - that is, when the cone of proximal normals and the Clarke normal cone may not coincide. In this case the right-hand side of (\ref{projection}) may not be upper semicontinuous. That is why we will consider the problem (\ref{projection}) when the the graph of its right-hand side is the closure of the graph of the original projection mapping:
\begin{equation}\label{projection_pr_stationary}
\begin{array}{l} \dot x(t) \in  G(F(x(t)),x(t))\\ x(0) = x_0
\in C\\ x(t) \in C ,\end{array}
\end{equation}
where $C$ is a closed subset of $\mathbb{R}^n$ and the multivalued mapping $G$ is obtained by closing the graph of
$(d,x)\mapsto Proj_{T_C(x)}(d)$ (here $T_C(x)$ is the Bouligand tangent cone of $C$ at $x$).
We are going to refer to (\ref{projection_pr_stationary}) as "projection process". This problem is closely related to a sweeping process with perturbation, where the normal cone is assumed to be the cone of limiting normals:
\begin{equation}\label{sweep_new_stationary}
\begin{array}{l} \dot x(t) \in -N_{C} (x(t)) +F(x(t))\\ x(0) = x_0
\in C\\ x(t) \in C,\end{array}
\end{equation}
where $C$ is a closed set in $\mathbb{R}^n$ and $N_{C} (x)$ is the cone of limiting normals to $C$ at $x$. Let us point out that now the right-hand sides of (\ref{projection_pr_stationary}) and (\ref{sweep_new_stationary}) are upper semicontinuous mappings with possibly nonconvex values which makes their investigation a lot harder. Proposition 6.7 on p.219 from \cite{RockWets} (see also Lemma 3.8 from \cite{GeorgRibarska}) shows that the right-hand side of (\ref{sweep_new_stationary}) contains the right-hand side of (\ref{projection_pr_stationary}). We do not know whether it is true that (\ref{sweep_new_stationary}) has a solution for arbitrary closed set $C$ even in the case when the perturbation $F(x)$ is single-valued and constant. We give an example showing that it is possible the projection process (\ref{projection_pr_stationary}) (with constant single-valued perturbation) to have no solution while the sweeping process (\ref{sweep_new_stationary}) to admit one. Thus, the problems (\ref{projection_pr_stationary}) and (\ref{sweep_new_stationary}) are no longer equivalent.

In order to include the case of moving obstacles in the crowd motion model, as well as to have a closer relation to the classical sweeping process, we are going to study a more general sweeping process

\begin{equation}\label{sweep_new}
\begin{array}{l} \dot x(t) \in -N_{C(t)} (x(t)) +F(x(t), t)\\ x(0) = x_0
\in C(0)\\ x(t) \in C(t),\end{array}
\end{equation}
where $C(t)$ is a moving closed set in $\mathbb{R}^n$ and $N_{C(t)} (x)$ is the cone of limiting normals to the set $C(t)$ at the point $x$, and a more general projection process

\begin{equation}\label{projection_process}
\begin{array}{l} \dot x(t) \in Pr \ G(F(x(t), t), x(t), t)\\ x(0) = x_0
\in C(0)\\ x(t) \in C(t) ,\end{array}
\end{equation}
where $C(t)$ is a moving closed set in $\mathbb{R}^n$, $Pr: \mathbb{R}^{n+1}\to \mathbb{R}^n$ assigns to a $(n+1)$-dimensional vector the vector of its first $n$ coordinates and the multivalued mapping $G$ is obtained by closing the graph of
$$(d,x,t)\mapsto Proj_{T_K(x,t)\cap \{ t=1\}}(d), \ \ K:= \{ (x,t)\in R^{n+1}: \ x\in C(t)\}.$$
This is the natural way of extending (\ref{projection_pr_stationary}) from a stationary set to the nonstationary situation. See Example \ref{ex} and the last section for further discussion on this point.

We impose some additional conditions on the geometry of the sets involved and on the perturbation in order to be able to prove some existence results for the sweeping process and for the projection process. The condition imposed on the sets is defina\-bi\-lity in some o-minimal structure. Definability implies the existence of a Whitney stratification (see Definition \ref{Whitney} below). Dynamical systems with stratified domains have been studied in \cite{BreHong}, \cite{BarWol} and many others. Our problem could be considered as an weak invariance problem on a stratified domain, but we do not impose any conditions on the subdomains (while proximal smoothness and wedgeness are assumed in \cite{BarWol}). Moreover, in both papers an additional Structural Condition (SC) on the dynamics is assumed, while we prove the respective property for a specific regularization.

 Under the definability hypothesis an existence result for (\ref{sweep_new_stationary}) has been obtained in the paper \cite{GeorgRibarska} provided the perturbation is single-valued and continuous. In the same paper the existence of solution to (\ref{sweep}) is proved if the multimapping $C(t)$ is definable and Lipschitz (with respect to the Hausdorff distance). Now we are able to extend these results to the problem (\ref{sweep_new}) under the same assumptions. In fact, the definability assumption on $C$ (and continuity and single-valuedness of the perturbation) yields the existence of solution of the projection process (\ref{projection_process}) (and therefore of (\ref{projection_pr_stationary})) as well. The existence of solution to (\ref{projection_pr_stationary}) has been announced in \cite{GeorgRibarska} in a remark, but the proof there is not complete. Thus, Theorem \ref{thProj} is new (with respect to \cite{GeorgRibarska}), even in the stationary case.

The paper is organized as follows. Preliminaries on o-minimal structures and some basic definitions of nonsmooth analysis are in Section 2. An example where the projection process (\ref{projection_pr_stationary}) has no solution appears in Section 3. The main existence results for (\ref{sweep_new}) and (\ref{projection_process}) are in Section 4. An application to crowd motion model and some additional motivation of (\ref{projection_process}) are gathered in Section 5.

\section{Preliminaries}

Definable and tame sets, functions, and mappings are a product of
model theory and algebraic geometry; they are the main concepts of
the theory of so-called "o-minimal structures" that has been
actively developing during last $20 -– 25$ years (c.f., e.g.
\cite{Coste}, \cite{vanDries}). Applications of this theory to
optimization problems are becoming increasingly popular because the
classes of sets and mappings involved are, on one hand side, broad
enough to encompass a big part of the important applications and, on
the other hand side, small enough to avoid "pathologies" like
fractals.

\begin{definition} A structure (expanding the real closed field R) is a
collection ${\cal S} = ({\cal S}^n)_{n\in N}$, where each ${\cal
S}^n$ is a set of subsets of the  space $R^n$, satisfying the
following axioms:
\\ 1. All algebraic subsets of $R^n$ are in ${\cal S}^n$. (Recall that an algebraic set
is a subset of $R^n$ defined by a finite number of polynomial
equations
$$P_1(x_1, . . . , x_n) = . . . = P_k(x_1, . . . , x_n) = 0 .)$$
 2. For every $n$, ${\cal S}^n$ is a Boolean subalgebra of the powerset $2^{R^n}$ of $R^n$.
\\ 3. If $A \in {\cal S}^m$ and $B \in {\cal S}^n$, then $A \times B \in {\cal S}^{m+n}$.
\\ 4. If $p : R^{n+1} \longrightarrow R^n$ is the projection on the first $n$ coordinates
and $A \in {\cal S}^{n+1}$, then $p(A) \in {\cal S}^n$.

The elements of ${\cal S}^n$ are called the definable subsets of
$R^n$. The structure $\cal S$ is said to be o-minimal if, moreover,
it satisfies: \\ 5. The elements of ${\cal S}^1$ are precisely the
finite unions of points and intervals.
\end{definition}

In this work we always assume that the closed field $R$ coincides
with the field of the real numbers. Standard examples of o-minimal
structures are the semialgebraic sets (finite unions of sets defined
by finitely many algebraic equalities and inequalities), globally
subanalytic sets (this class contains all bounded sets which are
finite unions of sets defined by finitely many analytic equalities
and inequalities).

Any o-minimal structure enjoys magnificent stability properties,
e.g. the closure and the interior of a definable subset of $R^n$ are
definable; the image of a definable set by a definable map (i.e.
whose graph is a definable set) is definable. In fact, every
"reasonably" defined set, that is the definition uses finite
combination of quantifiers, is definable (provided quantified
variables range over definable sets).  The following definition and
the subsequent theorem make this precise:

\begin{definition} A first-order formula is constructed
recursively according to the following rules.\\ 1. If $P$ is a
polynomial of $n$ variables, then $$P(x_1, \dots ,
x_n)=0 \mbox{ and } P(x_1, \dots , x_n)>0$$ are first-order formulas. \\
2. If $A$ is a definable subset of $R^n$, then $x\in A$ is a
first-order formula. \\
3. If $\Phi(x_1, \dots , x_n)$ and $\Psi(x_1, \dots , x_n)$ are
first-order formulas, then "$\Phi$ and $\Psi$", "$\Phi$ or $\Psi$",
"not $\Phi$", "$\Phi \Longrightarrow \Psi$" are first-order
formulas. \\
4. If $\Phi(y,x)$ is a first-order formula (where $y=(y_1, \dots ,
y_p)$ and $x=(x_1, \dots , x_n)$) and $A$ is a definable subset of
$R^n$, then "$\exists  \ x\in A: \ \Phi(y,x)$" and "$\forall \  x\in A: \
\Phi(y,x)$" are first-order formulas.
\end{definition}

\begin{theorem}If $\Phi(x_1, \dots , x_n)$ is a first-order formula, the
set of all vectors $(x_1, \dots , x_n)$ in $R^n$ which satisfy
$\Phi(x_1, \dots , x_n)$ is definable.
\end{theorem}

The topology of definable sets is also "tame". Any definable subset
of $R^n$ has a partition into finitely many definable subsets each
of which is definably arcwise connected (see Theorem 3.9 in
\cite{Coste}). Let us recall also the Monotonicity theorem (Theorem
2.1, \cite{Coste}): Let $f:(a,b)\longrightarrow R$ be a definable
function. Then there exists a finite subdivision of $(a,b)$ such
that on each open interval in this subdivision $f$ is continuous and
either constant or strictly monotone. Moreover, for every $k\in N$
the function $f$ is piecewise $C^k$-smooth.

The most important (from our point of view) feature of the sets
belonging to some o-minimal structure is the highly nontrivial fact
that they admit a regular $C^k$-Whitney stratification for any $k\in
N$:

\begin{definition} \label{Whitney}
 Let $A\subset R^n$ and $k\in N$. We say that $A$ admits a
 regular $C^k$-Whitney stratification if there exists a finite
partitioning of $A$ into $C^k$ manifolds $\{ M_i\}_{i=1}^{i_0}$
(called strata) such that

- if $M_j \cap \overline{M_i} \not = \emptyset$, then $M_j \subset
\overline{M_i}$;

- if $x \in M_j$ and $x_k \in M_i$ converge to $x$ as
$k\longrightarrow \infty$, then $T_x M_j$, the tangent space to
$M_j$ at $x$, is contained in the lower limit of $T_{{x_k}} M_i$.
\end{definition}

Lower limit ($\liminf$) and upper limit ($\limsup$) of sets are
understood in Kuratowski sense in this note. We will denote by
$B_r(x)$ (resp. $\overline{B}_r(x)$) the open (resp. closed) ball in
$\mathbb{R}^n$ with center $x$ and radius $r$.

A multivalued map is called \textit{definable} if its graph is a definable set.

Let us recall the basic concepts of normal cones to a closed set $C
\subset \mathbb{R}^n$ at
some point $x\in C$. A vector $\zeta \in \mathbb{R}^n$ is said to be a {\it
proximal normal} to $C$ at $x\in C$ if there exists a positive real
$t$ such that the metric projection of the point $x+t\zeta$ on $C$
coincides with $x$. The cone of all proximal normals to $C$ at $x$
is denoted by $\widehat N_C(x)$. Closing the graph of the mapping
$x\mapsto \widehat N_C(x)$ we obtain the mapping $x\mapsto  N_C(x)$
assigning to each point $x\in C$ the {\it cone of  limiting normals}
 $N_C(x)$ to $C$ at $x$. The cone
$\overline{\mbox{co}} \ N_C(x)$ is said to be the {\it Clarke normal
cone} to $C$ at $x$.

\section{Example}

We are going to construct a closed set $C$ in $\mathbb{R}^2$ such that the projecting process (\ref{projection_pr_stationary}) with $F\equiv (1,1)$ starting at the
origin has no solution, but the respective sweeping process (\ref{sweep_new_stationary}) admits one.

The construction of the closed set $C$ follows a Cantor-like
procedure. Let us us fix the positive reals $a\not =b$ with $a+b=1$.
 We start by building the "base cell" $K$.
 $$K:= \left\{ (x,0): \ x\le a \right\} \cup
\left\{ (x,b): \ x\le a\right\}\cup $$ $$\cup
\left\{ (0,y): \  y\le 0 \right\} \cup \left\{
(a,y): \ y\le b \right\} \ .$$
The so defined set has the property that at every point of $K$
except for the vertex $(a,b)$ the cone of limiting normals is contained in the set
$N:=\{ (t,0): t\in \mathbb{R} \} \cup \{ (0,t): t\in \mathbb{R} \}$ 
of vectors along the coordinate axes and the projection of the drift term $(1,1)$ on the Bouligand tangent cone to $K$ is either $(1,0)$ or $(0,1)$. The differential inclusions
(\ref{projection_pr_stationary}) and (\ref{sweep_new_stationary}) with $K$ instead of $C$  have  unique solution starting at any point $x_0\in K$ (and inevitably ending at the "upper vertex" $(a,b)$).

Let us fix $q\in (0,\frac{1}{3})$. We construct the set $C$ by
induction. We begin by placing the set $q {K}$ in the
middle of the segment $[(0,0), (a,b)]$, that is $$C_1:=
\frac{1-q}{2}(a,b) + q.{K}$$ is our set on the first
step. We put $\displaystyle{T_1:=\left(\frac{1-q}{2},
\frac{1+q}{2}\right)}$. To continue with the second step we consider
the set $\{ (ta, tb): \ t\in [0,1]\setminus T_1 \}$. It is
contained in the segment $[(0,0), (a,b)]$ and consists of two
intervals. To obtain $C_2$, we add to $C_1$ two copies of $q^2
{K}$ placed in the middle of these two intervals, that is we put
$$C_2:= C_1 \cup \left( \left( \frac{1-q}{4}- \frac{q^2}{2}\right) (a,b)
+ q^2 {K} \right) \cup \left( \left( \frac{3+q}{4}-
\frac{q^2}{2}\right) (a,b) + q^2 {K} \right)$$
$$\mbox{and } T_2:=T_1 \cup \left( \frac{1-q}{4}- \frac{q^2}{2},
\frac{1-q}{4}+ \frac{q^2}{2}\right) \cup \left( \frac{3+q}{4}-
\frac{q^2}{2}, \frac{3+q}{4}+ \frac{q^2}{2}\right) \ .$$ If we have
constructed $C_n$ for some positive integer $n$, we build $C_{n+1}$
by taking the union of $C_n$ and $2^{n}$ copies of the set $q^{n+1}
{K}$ placed in the middle of all $2^{n}$ closed intervals
in $\{ (ta, tb): \ t\in [0,1]\setminus T_n \}$ and define
$T_{n+1}$ accordingly.

Define $C$ to be the closure of the set
$\displaystyle{\bigcup_{n=1}^\infty C_n}$ and $T$ to be the union
$\displaystyle{\bigcup_{n=1}^\infty T_n}$. We built $C$ in such a
way that it preserves the property of $K$ that at every point of $C$
except for the vertex $(a,b)$  the 
cone of limiting normals is contained in the set $N$ of vectors along the coordinate
axes. Moreover, the projection of the drift term $(1,1)$ on the Bouligand tangent cone at any point of $C$
except for the vertex $(a,b)$ is contained in the three-element set $\{ (1,0),(0,1),(1,1)\}$.

We claim that there is no solution of the projection process (\ref{projection_pr_stationary}) with $x_0=(0,0)$. Indeed, let us denote by $\bar
y(\cdot)$ the line which on each cell coincides with the unique
solution of (\ref{projection_pr_stationary}) starting at the "lower vertex". Note that
our assumption $a+b=1$ yields that the time
interval, for which  $\bar y (t)$ belongs to a cell, exactly
corresponds to some interval in
$T=\displaystyle{\bigcup_{n=1}^\infty T_n}$.  Thus we may assume
 that  $\bar y (t)$ belongs to the segment
$[(0,0), (a,b)]$ for every $t\in [0,1]\setminus T$.  The Cantor set
$[0,1]\setminus T$ is of positive Lebesgue measure. More precisely,
$$meas\left([0,1]\setminus T\right) = \frac{1-3q}{1-2q} \ .$$
Sure, the sets $$\left\{ x_1\in [0,a] : x_1=at, \ t\in \left([0,1]\setminus T\right)\right\}$$ and $$\left\{ x_2\in [0,b] : x_2=bt,  t\in \left([0,1]\setminus T\right)\right\}$$ are of positive Lebesgue measure as well. 
As the coordinates of $\bar y (\cdot)$ are monotone increasing,
$\bar y (\cdot)$ is differentiable almost everywhere. Therefore the
set
$$A:=\left\{ t\in [0,1]\setminus T : \ \dot{\bar y}
(t) \mbox{ exists }\right\}$$ is of positive Lebesgue measure. It is
clear that $\dot{\bar y} (t)$  is collinear to $(a,b)$ for every
$t\in A$. 

Now let us assume that $y(t)$, $t\in [0,T']$ is a solution of the projection process (\ref{projection_pr_stationary}) with $x_0=(0,0)$. Then the line $\{ y(t): t\in [0,T']\}$ coincides with the line $\{ \bar y(t): t\in [0,T]\}$. Indeed, the solution of the projection process is unique on a cell (and the trajectory ends on the diagonal) and if a point has one of its coordinates in the Cantor set, but it is not on the diagonal, the only solution of the projection process starting at it is to glide to the diagonal because of a connectedness argument. Thus any deviation from the line $\{ \bar y(t): t\in [0,T]\}$ is impossible. Moreover, as $y(t)$ is absolutely continuous, the set of the values of $t\in [0,T']$, for which $y(t)$ is on the diagonal $[(0,0), (a,b)]$ is of positive Lebesgue measure. Thus on this set $\dot y(t)$ should be collinear to $(a,b)$. On the other hand,
$$G((1,1),x)\subset \{ (1,0),(0,1),(1,1)\}  \mbox{ for every }x\in C \ .$$
As $\{ (1,0),(0,1),(1,1)\} \cap \{ \lambda (a,b): \ \lambda\in \mathbb{R}\} = \emptyset$ (because $a\not = b$), this is a contradiction. 

With suitable change of time variable $\bar y(t)$ is a solution to the sweeping process (\ref{sweep_new_stationary}), because $((1,1)-N)\cap \{ \lambda (a,b): \ \lambda\in \mathbb{R}\} \not =\emptyset$.

\section{Main existence results}

\begin{theorem} \label{main}
Let the multivalued mapping $C: [0,T] \rightrightarrows \mathbb{R}^n$ be
Lipschitz (with respect to the Hausdorff distance), definable in some o-minimal structure and let its values $C(t)$ be nonempty and
compact. Let the mapping $d:\mathbb{R}^{n+1} \to \mathbb{R}^n$ be continuous. Then the sweeping process
\begin{equation} \label{sweepPert}
\begin{array}{l}
    \dot x(t) \in d(x(t), t) -N_{C(t)} (x(t))\\
    x(0) = x_0 \in C(0)\\
    x(t) \in C(t) \ ,
\end{array}
\end{equation}
where  $N_{C(t)} (x(t))$ denotes the cone of limiting normals to $C(t)$ at $x(t)$, has a solution.
\end{theorem}

\begin{proof}
\textit{\textbf{Step I.} Increasing the dimension by the time variable}

Let us denote the graph of $C$ by $$K:= \{(x,t)\in \mathbb{R}^{n+1}: \ t \in [0,T] , \ x\in C(t)\} \, .$$ The assumptions of the theorem yield that $K$ is definable and therefore it admits a regular Whitney stratification $\{ S_i\}_{i\in I}$.
We denote by $Pr: \mathbb{R}^{n+1}\to \mathbb{R}^n$ the projection on first $n$ coordinates, by $Pr_t: \mathbb{R}^{n+1}\to \mathbb{R}$ the projection on the $(n+1)$-th coordinate and by $L$ the Lipschitz constant of the multivalued mapping $C$. We denote $d_1(x, t) := (d(x,t),1)$, if $(x, t)\in K$ (i.e. $x \in C(t)$). It is clear that $d_1: K \to \mathbb{R}^{n+1}$ is a continuous mapping. Since $C$ is Lipschitz (with respect to the Hausdorff distance), it can easily be obtained that $K$ is compact.

Lemmata 3.6 and 3.7 from \cite{GeorgRibarska} yield that there exists a set $\hat{T} \subset [0,T] \, ,$ which is at most countable and such that $N_{C(t)}(x) \equiv Pr (N_K(x,t))$ for every $x\in C(t)$ and $t\not \in \hat{T}$. This means that the existence of a solution $x(t), \, t \in [0, T]$ to (\ref{sweepPert}) is equivalent to the existence of a solution to
\begin{equation}\label{n+1}
\begin{array}{l}
    \dot y(t) \in d_1(y(t)) - (Pr N_K (y(t)), 0)\\
    y(0) \in C(0) \times \{0\}\\
    y(t) \in K \ ,
\end{array}
\end{equation}
where $y(t) = (x(t), t) \, .$

\textit{\textbf{Step II.} Constraining the right-hand side of (\ref{n+1})}

Let $S_{i_y}$ be the stratum (from the fixed Whitney stratification of $K$), to which $y$ belongs, and $T_y S_{i_y}$ be the tangent space to $S_{i_y}$. We define the multivalued mapping
$$V(y):= d_1(y) - (Pr (N_K(y)),0) \cap \left( L(y) \overline{B}\right) \ \subset \mathbb{R}^{n+1} \ , $$
where
$$L(y):= \left\{
\begin{array}{ll}
    \mbox{dist}(d_1(y),T_{y}S_{i_{y }} \cap \{t=1\}) \, , \\ \quad \quad \quad \quad \mbox{ if } \mbox{dist}(d_1(y),T_{y}S_{i_{y }} \cap \{t=1\}) \le L + \|Pr(d_1(x, t))\|\\
    L + \|Pr(d_1(x, t))\| \, , \\ \quad \quad \quad \quad \mbox{ if } \mbox{dist}(d_1(y),T_{y}S_{i_{y }} \cap \{t=1\}) > L + \|Pr(d_1(x, t))\| \\
                         \quad \quad \quad \quad \mbox{ or } T_{y}S_{i_{y }} \cap \{t=1\} = \emptyset \, .
\end{array}\right.$$
The mapping $V$ is defined correctly on $K$.

Let us check that $V$ is upper semicontinuous. Since $K$ is compact and $d_1$ is continuous, this is equivalent to proving the closedness of its graph. Let $y_m=(x_m, t_m)\to y_0=(x_0, t_0)$ and $v_m \to v_0$, $v_m\in V(y_m)$. Because $d_1$ is continuous and $Pr N_K$ is upper semicontinuous, from $v_m \in d_1(y_m) - (Pr (N_K(y_m)),0)$ it follows that $v_0 \in d_1(y_0) - (Pr (N_K(y_0)),0).$

We denote by $S_j$, $j\in J$, all strata such that $S_{i_{y_0}}\subset \overline{S_j}$. The sequence $\{y_m\}_{m=1}^\infty$ can be split into finitely many subsequences (or finite sets), such that any of them is contained in one stratum $S_j$ for some $j\in J$ or in $S_{i_{y_0}}$. Thus, without loss of generality, we may assume that $\{ y_{m}\}_{m=1}^\infty$ is contained in one stratum. If it is $S_{i_{y_0}}$, obviously $ v_0$ belongs to $V(x_0)$. If $\{ y_{m}\}_{m=1}^\infty$ is contained in $S_j$ for some $j\in J$, in order to conclude that $v_0\in V(y_0)$, it remains to prove that $L(y_0) \ge \limsup_{m\to \infty} L(y_m)$.

Since $L(y) = \min(L + \|Pr(d_1(y))\|, \mbox{dist}(d_1(y),T_{y}S_{i_{y }} \cap \{t=1\}))$, it is enough to check
\begin{equation}\label{dist}
    \mbox{dist}(d_1(y_0),T_{y_0}S_{i_{y_0 }} \cap \{t=1\}) \ge \limsup_{m\to \infty} \mbox{dist}(d_1(y_0),T_{y_{m }}S_{ j} \cap \{t=1\}) \, ,
\end{equation}
because $$\mbox{dist}(d_1(y_0),T_{y_{m}}S_{j}\cap \{t=1\}) \ge \mbox{dist}(d_1(y_m),T_{y_{m}}S_{j}\cap \{t=1\}) - \| d_1(y_m)-d_1(y_0)\|$$
would yield
$$\limsup_{m\to \infty} \mbox{dist}(d_1(y_0),T_{y_{m}}S_{  j}\cap \{t=1\}) \ge \limsup_{m\to \infty} \mbox{dist}(d_1(y_m),T_{y_{m}}S_{  j}\cap \{t=1\}) \, . $$

Let us denote $A_m = T_{y_{m }}S_{  j} \cap \{t=1\}$, $m\in \mathbb{N}$ and $A = T_{y_{0 }}S_{  j_{y_0}} \cap \{t=1\}.$ According to Corrolary 4.7 on p.113 in \cite{RockWets}, it would suffice to prove $A \subset \liminf_{m\to\infty} A_m$. Let $a \in A$ and $U \subset \mathbb{R}^{n+1}$ be an open neighbourhood of $a$. Hence, there exist $V \subset \mathbb{R}^{n}$ -- an open neighbourhood of $Pr(a)$ and $\varepsilon > 0, \ \varepsilon < \frac{1}{2}$, such that $2V\times (1-\varepsilon,1+\varepsilon) \subset U$. Then $ \tilde{U}=V\times (1-\varepsilon,1+\varepsilon) \subset \mathbb{R}^{n+1}$ is an open neighbourhood of $a$. From the regularity of stratification: $T_{y_0}S_{i_{y_0}} \subset \liminf_{m\to \infty} T_{y_{m }}S_{j}$ and therefore $\tilde{U} \cap T_{y_{m}}S_{j} \neq \emptyset$ for every $m \ge m_0$. Let $a_m = (b_m, t_m) \in \tilde{U} \cap T_{y_{m}}S_{j}$ for every $m \ge m_0$. Let us consider the sequence $a_{m}' = (\frac{1}{t_m} b_m, 1), \ m \ge m_0$. We know that 
$$\frac{1}{t_m} b_m \in \frac{1}{t_m}V \subset \frac{1}{1-\varepsilon}V \subset 2V$$
 and therefore $a_{m}' \in (2V\times (1-\varepsilon,1+\varepsilon)) \cap A_m \subset U \cap A_m$ for every $m \ge m_0$. We have obtained that $A \subset \liminf_{m\to\infty} A_m$, which verifies (\ref{dist}) and finishes the proof that $V$ is upper semicontinuous.

\textit{\textbf{Step III.} Proving that the inclusion with constrained convexified right-hand side has a solution}

Let us consider the constrained differential inclusion with convex right-hand side
\begin{equation} \label{convexified}
\begin{array}{l}
    \dot y(\tau) \in \overline{\rm co}\left( V(y(\tau))\right)\\
    y(0) =(x_0,0) \in C(0)\times \{ 0\}\\
    y(\tau) \in K \ .
\end{array}
\end{equation}
We are going to prove it has a solution by checking a sufficient condition for the weak invariance of $K$ with respect to the inclusion and then applying Theorem 2.10 on p.193 from \cite{Clarke} (see also the viability theorem in \cite{Veliov}).

It is straight-forward to check that the right-hand side of the inclusion is upper semicontinuous. Obviously, the right-hand side of the inclusion is nonempty convex compact valued, so it remains to justify the following

\begin{proposition} \label{claim} For every $y_0=(x_0,t_0) \in K$ there exists a feasible velocity $v_0\in V(x_0,t_0)$, such that $v_0 \in T_K(y_0)$.
\end{proposition}

To this end, we are going to use Lemma 3.5 from \cite{GeorgRibarska} and the following lemma.

\begin{lemma} \label{catchUpLemma}
Let $y_0\in K$. Then any vector $v_0$ belonging to the metric projection of $d_1(y_0)$ on $T_K(y_0) \cap \{t=1\}$ has the property that $d_1(y_0) - v_0\in (Pr(N_K(y_0), 0)$.
\end{lemma}

\begin{proof}[Proof of the Lemma.] According to Proposition 6.27(a) on p.219 from \cite{RockWets}: $N_{T_K(y_0)}(0) = \bigcup_{v \in T_K(y_0)} N_{T_K(y_0)}(v) \subset N_K(y_0)$. Hence, it is enough to prove that $d_1(y_0) - v_0\in (Pr(N_{T_K(y_0)}(v_0)), 0).$

We denote $T := T_K(y_0)$ and $T_1 := T \cap \{t=1\}$. Let us examine the $n$-dimentional ball
$$A := \overline{B} \left(d_1(y_0), \mbox{dist}(d_1(y_0), T_1)\right) \cap \{t=1\}$$
and the cone $C := \{y \in \mathbb{R}^{n+1} \ | \ y = \alpha(y_1-y_0), \ \alpha >0, \ y_1 \in A\} \, .$
Since $\mbox{bdry} \, C \: \setminus \: \{\mathbf{0}\}$ is a $\mathit{C}^2$ surface, all normal vectors to it are proximal normals as well (Proposition 1.9 on p.26 from \cite{Clarke}).

Also, we know that
$$T \cap C = \{y \in \mathbb{R}^n \ | \ y = \alpha v_1, \ \alpha \ge 0, \ v_1 \in Proj_{T_1} d_1(y_0)\} \, .$$
Thus we have $v_0 \in T \cap C$. Let $\xi$ be a normal vector to $\mbox{bdry} \, C$ at $v_0$, pointing to the inside of the cone. Then $\xi$ is a proximal normal to $\overline{\mathbb{R}^{n+1} \setminus C}$ at $v_0$. Hence, there exists $r>0$, such that
$$\overline{B}_r(v_0+r\xi) \cap (\overline{\mathbb{R}^{n+1} \setminus C}) = \{v_0\} \ .$$

Since $v_0 \in T \subset \overline{\mathbb{R}^{n+1} \setminus C}$, $\, \overline{B}_r(v_0+r\xi) \cap T = \{v_0\}$ and therefore $\xi \in N_T(v_0)$. Hence $Pr \xi = \beta (d_1(y_0) - v_0), \ \beta > 0 \, .$ Because $\overline{B}_r(v_0+r\xi) \cap \{t=1\}$ is a $n$-dimensional ball with positive radius, whose intersection with $\overline{(\mathbb{R}^{n+1} \cap \{t=1\}) \setminus A}$ is $v_0$, we obtain that $Pr \xi = \beta (d_1(y_0) - v_0), \ \beta > 0.$ Since $N_T(v_0)$ is a cone, $d_1(y_0) - v_0 \in (Pr(N_{T}(v_0)), 0)$.
\end{proof}

\begin{proof}[Proof of Proposition \ref{claim}.] Let $v_0 = (w_0, 1)$ belong to the metric projection of $d_1(y_0)$ on $T_K(y_0) \cap \{t=1\}$. From Lemma \ref{catchUpLemma}, it follows that $d_1(y_0) - v_0\in (Pr(N_{T_K(y_0)}(0)), 0).$ If we justify that $\|d_1(y_0) - v_0\| \le L(y_0)$, the Proposition would be proven, because then $v_0$ would belong to $V(y_0)$.

We know that $L(y_0) = \min(\mbox{dist}(d_1(y_0),T_{y}S_{i_{y }} \cap \{t=1\}), \|Pr(d_1(y_0))\| + L)$. According to Example 6.8 on p.203 from \cite{RockWets}, $T_{y_0}S_{i_{y_0}}$ coincides with the Bouligand tangent cone to $S_{i_{y_0}}$. From $S_{i_{y_0}} \subset K$ we obtain that 
$$T_{y_0}S_{i_{y_0}}  \cap \{t=1\} \subset T_K(y_0)  \cap \{t=1\}$$
and therefore:
$$\|d_1(y_0) - v_0\| = \mbox{dist}(d_1(y_0),T_K(y_0) \cap \{t=1\}) \le \mbox{dist}(d_1(y_0), T_{y_0}S_{i_{y_0}} \cap \{t=1\}) \, .$$
On the other hand
$$\|d_1(y_0) - v_0\| \le \|d_1(y_0) - v_+\| = \|Pr(d_1(y_0)) - w_+\| \le \|Pr(d_1(y_0))\| + L \, ,$$
where $v_+=(w_+,1) \in T_K(y_0)$ is from Lemma 3.5 from \cite{GeorgRibarska}.

We have proven that $\|d_1(y_0) - v_0\| \le L(y_0)$ and therefore $v_0\in V(y_0)$.
\end{proof}

\textit{\textbf{Step IV.} Proving that each solution of the inclusion with constrained convexified right-hand side is a solution to the inclusion with constrained nonconvexified right-hand side}

Let $y(t)$, $t\in [0,T]$, be a solution to the differential inclusion with constrained convexified right-hand side (\ref{convexified}). We are going to prove that $\dot y(t)\in V(y(t))$ for almost all $t\in [0,T]$.

Let $t_0\in [0,T]\setminus \hat T$ be such that $\dot y(t_0)$ exists, $\dot y(t_0)\in \overline{\rm co} \ V(y(t_0))$ and $t_0$ is a cluster point of $$T' := \{ t\in [0,T]: \ y(t)\in S_{i_{y(t_0)}}\} \ .$$
Since $\hat T$ is at most countable,  $y(t)$ is a solution to (\ref{convexified}) and because the set $\{ t\in [0,T]: \ y(t)\in S_{i}\} $ may have only countably many isolated points for every $i\in I$ and $I$ is finite, almost all elements of $[0,T]$ satisfy the above requirements.

Then $\dot y(t_0)$ must belong to $T_{y(t_0)}S_{i_{y(t_0)}} \, ,$ thus $T_{y(t_0)}S_{i_{y(t_0)}}\cap \ \overline{\rm co} \ V(y(t_0))\not = \emptyset$. But $V(y(t_0))$ and therefore $\overline{\rm co}V(y(t_0))$ are contained in the set $\{ (w,1): \ \|d_1(y_0) - w\| \le L(y(t_0)) \} \, .$
Also,
$$\{ (w,1): \ \| Pr d_1(y_0)-w\| \le L(y(t_0)) \} \ \cap \ T_{y(t_0)}S_{i_{y(t_0)}} \subset \left\{ Proj_{T_{y(t_0)}S_{i_{y(t_0)}} \cap \{t=1\}} d_1(y_0)\right\} \, .$$

From $T_{y(t_0)}S_{i_{y(t_0)}}\cap \ \overline{\rm co} \ V(y(t_0))\not = \emptyset$ we obtain that
$$T_{y(t_0)}S_{i_{y(t_0)}}\cap \ \overline{\rm co} \ V(y(t_0))= \left\{ Proj_{T_{y(t_0)}S_{i_{y(t_0)}} \cap \{t=1\}} d_1(y_0)\right\} \, .$$

If we assume that $T_{y(t_0)}S_{i_{y(t_0)}}\cap \ V(y(t_0))= \emptyset$, then
\begin{equation}
\begin{array}{l}
 \mathbf{0} \notin V(y(t_0)) - Proj_{T_{y(t_0)}S_{i_{y(t_0)}} \cap \{t=1\}} d_1(y(t_0)) \subset  \\
\subset \left(d'(y(t_0)) - (Pr (N_K(y(t_0))),0) \right) \cap \left( \overline{B}_{L(y(t_0))} \left( d'(y(t_0)) \right) \right) \, ,
\end{array}
\end{equation}
where $d'(y(t_0)) := d_1(y(t_0)) - Proj_{T_{y(t_0)}S_{i_{y(t_0)}} \cap \{t=1\}} d_1(y(t_0))$ and $\|d'(y(t_0))\| \le L(y(t_0)) \, .$

Thus, there exists $\varepsilon > 0$, such that
$$B_\varepsilon (0) \cap \left( V(y(t_0))- Proj_{T_{y(t_0)}S_{i_{y(t_0)}} \cap \{t=1\}} d_1(y(t_0))\right)=\emptyset \, .$$

From the uniform convexity of the ball in $\mathbb{R}^n$ it follows that there exist a continuous linear functional $\varphi$ and a positive real number $\alpha$, such that
$$ \mbox{diam}\left\{ y\in \overline{B}_{L(y(t_0))} \left( d'(y(t_0)) \right): \ \varphi (y)< \alpha \right\} < \varepsilon \, .$$
Then
$$ V(y(t_0))- Proj_{T_{y(t_0)}S_{i_{y(t_0)}} \cap \{t=1\}} d_1(y(t_0)) \subset \left\{y\in \mathbb{R}^n : \ \varphi(y) \ge \alpha \right\}$$
and therefore
$$ \overline{\mbox{co}} \ V(y(t_0))\subset Proj_{T_{y(t_0)}S_{i_{y(t_0)}} \cap \{t=1\}} d_1(y(t_0)) +\left\{y\in \overline{B}_{L(y(t_0))} \left( d'(y(t_0)) \right) : \ \varphi(y) \ge \alpha \right\} \, ,$$
which is a contradiction to $Proj_{T_{y(t_0)}S_{i_{y(t_0)}} \cap \{t=1\}} d_1(y(t_0)) \in \overline{\mbox{co}} \  V(y(t_0)) \, .$

We obtain that $$T_{y(t_0)}S_{i_{y(t_0)}}\cap  \ V(y(t_0))= \left\{ Proj_{T_{y(t_0)}S_{i_{y(t_0)}} \cap \{t=1\}} d_1(y(t_0)) \right\} \, ,$$ which proves $\dot y(t_0) \in V(y(t_0))$. Therefore $\dot y(t)\in V(y(t))$ for almost all $t\in [0,T]$ and $y$ is an absolutely continuous mapping with values in $K$.

It is clear that $y(t)=(x(t), t)$ for all $t\in [0,T]$, since the last coordinate of the right-hand side is the constant $1$. Then $x(t)\in C(t)$, because $y(t)\in K,$ $t\in [0,T]$. If $t\in [0,T]\setminus \hat T$ and $\dot y(t)\in V(y(t))$, we obtain that
$$\dot x(t) \in Pr(d_1(x(t), t)) - Pr(N_K(x(t),t)) = d(x(t))-N_{C(t)}(x(t)) \, .$$

Therefore $x(t)$, $t\in [0,T]$, is a solution of the sweeping process (\ref{sweepPert}).
\end{proof}

\vspace{0.5cm}

\begin{theorem} \label{thProj}
Let the multivalued mapping $C: [0,T] \rightrightarrows \mathbb{R}^n$ be
Lipschitz (with respect to the Hausdorff distance), definable in some o-minimal structure, its values $C(t)$ be nonempty and
compact and the mapping $d:\mathbb{R}^{n+1} \to \mathbb{R}^n$ be continuous. Then the projection process
\begin{equation*}
\begin{array}{l} \dot x(t) \in Pr \ G(d(x(t), t), x(t), t)\\ x(0) = x_0
\in C(0)\\ x(t) \in C(t) \, ,\end{array}
\end{equation*}
where the multivalued mapping $G$ is obtained by closing the graph of
$$(d,x,t)\mapsto Proj_{T_K(x,t)\cap \{ t=1\}}(d), \ \ K:= \{ (x,t)\in R^{n+1}: \ x\in C(t)\} \, ,$$
has a solution.
\end{theorem}

\begin{proof} Let $y(t) \, , t \in [0,T]$ be the solution of the sweeping process (\ref{sweepPert}), constructed in the proof of Theorem \ref{main}. We are going to use the same notations as in the previous proof.

Let $\dot y(\bar t) \, , \bar t \in [0,T]$ exist and belong to 
$$T_{y(\bar t)}S_{i_{y(\bar t)}}\cap  \ V(y(\bar t)) = \left\{ Proj_{T_{y(\bar t)}S_{i_{y(\bar t)}} \cap \{t=1\}} d_1(y(\bar t)) \right\} \, . $$
Note that it is true for almost all $t$, according to the proof of the above theorem. Then $\dot y(\bar t)=d_1(y(\bar t))-\eta_0$ and $\eta_0 \in (Pr N_K(y(\bar t)), 0)$. Let $\zeta_0 = (\eta_0, t_0) \in N_K(y(\bar t))$. Therefore there exist $y_m\longrightarrow_{m\to \infty} y(\bar t)$, $\{y_m\}\subset K$, $\zeta_m = (\eta_m, t_m) \longrightarrow_{m\to \infty} \zeta_0$ with $\zeta_m \in \widehat{N}_K(y_m)$. Let us put $$v_m:= d_1(y_m)-\eta_m \mbox{ and } w_m:= Proj_{T_K(y_m)\cap \{t=1\}}d_1(y_m).$$
It is clear that $v_m \longrightarrow d_1(y(\bar t))-\eta_0=\dot y(\bar t) =: v_0$ and $w_m \in G(y_m)$. If we can show
that $\| v_m - w_m\| \longrightarrow_{m\to \infty} 0$, we are done, because then $(y_m, w_m)\in Graph (G)$, $(y_m,w_m)\longrightarrow_{m\to \infty} (y(\bar t),\dot y(\bar t))$.

Let us denote by $S_j$, $j\in J$, all strata such that $S_{i_{y_0}}\subset \overline{S_j} \, ,$ where $y_0 := y(\bar t) \, .$ The sequence $\{y_m\}_{m=1}^\infty$ can be split into finitely many subsequences (or finite sets) such that any of them is contained in one stratum $S_j$ for some $j\in J$ or in $S_{i_{y_0}}$. Thus, without loss of generality, we may assume that $\{ y_{m}\}_{m=1}^\infty$ is contained
in one stratum. If it is $S_{i_{y_0}}$, clearly ${\rm dist} \left( v_m, T_{y_m}S_{i_{y_0}}\cap \{t=1\}\right)\longrightarrow_{m\to \infty} \mbox{dist}(v_0,T_{y_0}S_{i_{y_0 }} \cap \{t=1\}) =0 \, .$  If $\{ y_{k}\}_{m=1}^\infty$ is contained in $S_j$ for some $j\in J$, analogously to (\ref{dist}), we have
$$ \mbox{dist}(v_0,T_{y_0}S_{i_{y_0 }} \cap \{t=1\}) \ge \limsup_{m\to \infty} \mbox{dist}(v_m,T_{y_{m }}S_{ j} \cap \{t=1\}) \, ,$$
so ${\rm dist} \left( v_m, T_{y_m}S_{j}\cap \{t=1\}\right)\longrightarrow_{m\to \infty} 0$ again.
According to Example 6.8 on p.203 from \cite{RockWets}, $T_{y_0}S_{i_{y_0}}$ coincides with the Bouligand tangent cone to $S_{i_{y_0}}$. From $S_{j} \subset K$ we have that 
$$T_{y_m}S_{j}\cap \{t=1\} \subset T_K(y_m)\cap \{t=1\}$$
and therefore:
\begin{equation}\label{distvm}
{\rm dist} \left( v_m, T_K (y_m)\cap \{t=1\}\right) \le {\rm dist} \left( v_m, T_{y_m}S_{j}\cap \{t=1\}\right) \longrightarrow_{m\to \infty} 0 \, .
\end{equation}

We know that $\zeta_0 = \eta_0 + t_0 . e_t \, ,$ where $e_t = (\mathbf{0}, 1) \in \mathbb{R}^{n+1} \, ,$ and $\eta_0 \, \bot \, e_t$. Hence, $\langle d_1(y_0) - \zeta_0, \zeta_0 \rangle = \langle d_1(y_0) - \eta_0 - t_0 . e_t , \zeta_0 \rangle = \langle d_1(y_0) - \eta_0, \zeta_0 \rangle - t_0^2 \, .$
The regularity of the stratification implies $$N_{y(t_0)}S_{i_{y(t_0)}} \supset \limsup_{m\to \infty} N_{y_{m}}S_{j}$$ whenever $\{y_m\}_{m=1}^\infty \subset S_j$ tends to $y(t_0)$. Therefore, $N_K(y(t_0)) \subset N_{y(t_0)}S_{i_{y(t_0)}}$ (proximal normals to $K$ at any point belong to the normal space at the same point to the stratum, to which the point belongs -- Example 6.8 on p.203 in \cite{RockWets}). Therefore, $\langle d_1(y_0) - \eta_0, \zeta_0 \rangle = 0$ (because $d_1(y_0)-\eta_0 \in  T_{y_0}S_{i_{y_0}}$ and $\zeta_0 \in N_K(y(t_0)) \subset N_{y(t_0)}S_{i_{y(t_0)}}$) and so
\begin{equation}\label{scalar1}
\langle d_1(y_0) - \zeta_0, \zeta_0 \rangle = -t_0^2 \, .
\end{equation}
On the other hand, $\langle d_1(y_0) - \zeta_0, \zeta_0 \rangle = \langle d_1(y_0) - \eta_0, \eta_0 \rangle + t_0(1-t_0) \, .$ Therefore,
\begin{equation}\label{scalar2}
\langle d_1(y_0) - \eta_0, \eta_0 \rangle = -t_0 \, .
\end{equation}

Let $d_1(y_m)+\frac{t_m}{\|\eta_m\|^2} . \eta_m =\lambda_m . \eta_m +u_m$ where $\lambda_m \ge 0$ and $\langle u_m, \eta_m \rangle \le 0$ (moreover, let $\langle u_m ,  \eta_m  \rangle =0$ if $\lambda_m > 0$), i.e. 
$$d_1(y_m) =\lambda_m . \eta_m + u'_m  \, ,$$
where $u'_m = u_m - \frac{t_m}{\|\eta_m\|^2} . \eta_m \, .$
If $\lambda_m = 0$, then $\langle d_1(y_m) , \eta_m \rangle =\lambda_m \|\eta_m\|^2 + \langle u_m, \eta_m \rangle - t_m \le -t_m$ and therefore $\langle d_1(y_m) , \zeta_m \rangle \le 0$, which leads to $\langle d_1(y_0) , \zeta_0 \rangle \le 0$ as $m \to \infty$. From (\ref{scalar1}) we obtain that 
$$0 \ge \langle d_1(y_0), \zeta_0 \rangle = \langle d_1(y_0) - \zeta_0, \zeta_0 \rangle + \|\zeta_0\|^2 = -t_0^2 + \|\zeta_0\|^2 = \|\eta_0\|^2 \, , $$
hence $\eta_0 = \mathbf{0}$. Then $d_1(y_0) = v_0$ and from $T_{y_0}S_{i_{y_0}} \subset T_K(y_0)$ it follows that $d_1(y_0) \in T_K (y_0)\cap \{t=1\}$ and therefore $v_0 = d_1(y_0) = Proj_{T_K(y_0)\cap \{t=1\}}d_1(y_0) \, ,$ so $\dot y (t_0) \in G(y_0) \, .$

Thus, without loss of generality $\lambda_0 \not = 0$ and $\langle u_m ,  \eta_m  \rangle =0$. Then $\lambda_m \ne 0, \, \eta_m \ne \mathbf{0}$ as well for $m$ large enough and
$$ \lambda_m  = \frac{ \langle d_1(y_m), \eta_m \rangle}{\|\eta_m\|^2} +  \frac{t_m}{\|\eta_m\|^2} \longrightarrow_{m\to \infty} \frac{ \langle d_1(y_0), \eta_0 \rangle}{\|\eta_0\|^2} + \frac{t_0}{\|\eta_0\|^2} \, . $$
From (\ref{scalar2}) we have that 
$$ \frac{ \langle d_1(y_0), \eta_0 \rangle}{\|\eta_0\|^2} =\frac{ \langle d_1(y_0)-\eta_0, \eta_0 \rangle + \|\eta_0\|^2}{\|\eta_0\|^2}=-\frac{t_0}{\|\eta_0\|^2} +1 \, . $$
and we obtain $\lambda_m \longrightarrow_{m\to \infty} 1 \, .$
Thus
\begin{equation} \label{distuv}
u'_m - v_m =d_1(y_m) - \lambda_m . \eta_m -d_1(y_m) +\eta_m = (1-\lambda_m)\eta_m\longrightarrow_{m\to \infty} \mathbf{0} \, ,
\end{equation}
because $\|\eta_m\| \le L + \|d(y_m)\| \, ,$ which is bounded.
So, we shall be done if we prove that $u'_m - w_m \longrightarrow_{m\to \infty} \mathbf{0}$.

Because of $\langle w_m, \zeta_m\rangle \le 0$, we have $\langle w_m, \eta_m\rangle \le -t_m$ and therefore
\begin{equation} \label{ineq1}
\begin{array}{l}
\| w_m - d_1(y_m)\|^2=\langle w_m - \lambda_m  \eta_m -u'_m, w_m - \lambda_m  \eta_m -u'_m\rangle = \\
= \| w_m - u'_m\|^2 +\| \lambda_m  \eta_m\|^2 - 2\langle w_m - u'_m,\lambda_m  \eta_m \rangle =\\
=\| w_m - u'_m\|^2 +\| \lambda_m  \eta_m\|^2 - 2\lambda_m  (\langle w_m , \eta_m \rangle +t_m) \ge\\
\ge\| w_m - u'_m\|^2 +\| \lambda_m  \eta_m\|^2 \, .
\end{array}
\end{equation}

Putting $$\varepsilon_m:=\| (1-\lambda_m)\eta_m \|+{\rm dist}\left( v_m, T_K(y_m) \cap \{t=1\}\right)\longrightarrow_{m\to \infty}0$$ (due to (\ref{distuv}) and the construction of the solution), we obtain
\begin{equation} \label{ineq2}
\begin{array}{l}
\| w_m - d_1(y_m)\| = {\rm dist}\left( d_1(y_m), T_K(y_m) \cap \{t=1\}\right) \le \\
 \le \| d_1(y_m) - u'_m \| + {\rm dist}\left( u'_m, T_K(y_m) \cap \{t=1\}\right) \le \\
 \le \| \lambda_m  \eta_m\| + \|u'_m - v_m\| + {\rm dist}\left(v_m, T_K(y_m) \cap \{t=1\}\right) \le \\
 \le \| \lambda_m  \eta_m\| + \| (1-\lambda_m)\eta_m \| + {\rm dist}\left(v_m, T_K(y_m) \cap \{t=1\}\right) = \\
 = \| \lambda_m  \eta_m\| +\varepsilon_m
\end{array}
\end{equation}
using the triangle inequality and (\ref{distuv}).


Inequalities (\ref{ineq1}) and (\ref{ineq2}) yield
\begin{equation}
\begin{array}{l}
\| w_m - u'_m\| \le \sqrt{\| w_m - d_1(y_m)\|^2-\| \lambda_m  \eta_m\|^2} \le \\
 \le \sqrt{\left( \| \lambda_m  \eta_m\| +\varepsilon_m\right)^2-\| \lambda_m  \eta_m\|^2} \longrightarrow_{m\to \infty}0
\end{array}
\end{equation}
which concludes the proof of the theorem.
\end{proof}

\section{Application}

We are going to generalize the crowd motion model, presented in \cite{BerVen}, \cite{BerVen2} and apply the results from the previous section to it. The model rests on two principles. On the one hand, each individual
has a spontaneous velocity that he would like to have in the absence of other people or obstacles. On
the other hand, the actual velocity must take into account congestion and obstacles. Those two principles
lead to define the actual velocity field as the projection of the spontaneous velocity
over the set of admissible velocities (regarding the non-overlapping constraints and the obstacles).

Let us quickly recall the setting of the model without obstacles. A crowd with $N$ people is considered. They are identified to rigid disks with the same radius $r$ (for convenience). The centre of the $i$-th disk is denoted by $x_i \in \mathbb{R}^2$. Since overlapping is forbidden, the vector of positions $\mathbf{x} = (x_1, x_2, \ldots, x_N) \in \mathbb{R}^{2N}$ has to belong to the "set of feasible configurations", defined by
$$ C_0 := \left\{ \mathbf{x} \in \mathbb{R}^{2N} \ : \ D_{ij}(\mathbf{x}) \ge 0, \ i \neq j \right\} \, , $$
where $D_{ij}(\mathbf{x}) = \|x_i - x_j\| - 2r$ is the signed distance between disks $i$ and $j$.
In the model we take for granted the vector of spontaneous velocities
$$d(\mathbf{x}) = (d_1(\mathbf{x}), d_2(\mathbf{\mathbf{x}}), \ldots, d_N(\mathbf{x})) \in \mathbb{R}^{2N} \, .$$
The set of the admissible velocities is exactly $T_{C_0}(\mathbf{x}) \, ,$ the Bouligand tangent cone to $C_0$ at $\mathbf{x} \, :$
$$ T_{C_0}(\mathbf{x}) = \left\{ v \in \mathbb{R}^{2N} \ : \ \forall i<j  \ \ \ D_{ij}(\mathbf{x}) = 0 \ \ \Rightarrow \ \ \langle G_{ij}(\mathbf{x}), \, v \rangle \ge 0 \right\} \, , $$
where
$$G_{ij}(\mathbf{x}) = \nabla D_{ij}(x) = (0, \ldots , 0, -e_{ij}(x), 0, \ldots , 0, e_{ij}(x), 0, \ldots, 0) \in \mathbb{R}^{2N}$$
and $ e_{ij}(\mathbf{x}) = \frac{x_j - x_i}{\|x_j - x_i\|} \, .$ 

Our goal is to generalize the model in a way that movable and immovable obstacles are considered. We are going to allow obstacles that can be described by finitely many analytical equalities and inequalities. As we confine ourselves to compact sets, we remain in the frame of the o-minimal structure of subanalytic sets. This represents essentially all possible real-life cases.   

Let us examine a movable obstacle in the plain, defined by the analytic inequality $P(x, t) \le 0, \, x \in \mathbb{R}^2 \, .$ In order $\mathbf{x}(t)$ to be an admissible configuration at the moment $t$, every $\mathbf{x}' = (x'_1, x'_2, \ldots , x'_N) \in \mathbb{R}^{2N},$ such that $\|x'_i - x_i(t)\|^2 \le r^2$ for some $i = {1,\ldots N} \, ,$ should fulfill $P(x'_i, t) \ge 0 \, .$ Since $\|x'_i - x_i(t)\|^2 \le r^2,  i = {1,\ldots N}$ is an analytic inequality, 
$$\mathbf{x'} \in \mathbb{R}^{2N} \, : \, \|x'_i - x_i(t)\|^2 \le r^2,  i = {1,\ldots N} \Rightarrow P(x'_i, t) \ge 0$$
 is a first order formula and therefore the set of constraints due to $\{P(x, t) \le 0 \}$
\begin{equation}\label{admiss}
\begin{array}{l}
 C_{P(x, t)} := \{\mathbf{x} \in \mathbb{R}^{2N} :
  \forall \mathbf{x}' \in \mathbb{R}^{2N} \, \forall i = {1,\ldots N} \, (\|x'_i - x_i(t)\|^2 \le r^2 \Rightarrow P(x'_i, t) \ge 0) \,  \}
\end{array}
\end{equation}
is definable.

The obstacles that we examine can be described by finitely many inequalities of the type $P(x, t) \le 0$ (the equalities can be presented as 2 inequalities with opposite directions and for the immovable inequalities $P(x, t) := P(x)$). 

Let us examine $m$ movable and immovable obstacles, each of which described by $k$ inequalities of type $P(x, t) \le 0$. They can be presented as one obstacle in the following way:
$$S(x, t) := \bigcup_{i=1}^{m} \left( \bigcap_{j=1}^{k} \, \{ \, P_{ij}(x, t) \le 0 \, \} \right)  \, .$$

The set of the admissible configurations is
\begin{equation}
\begin{array}{l}
C(t) := \left\{ \mathbf{x} \in \mathbb{R}^{2N} \ : \ D_{ij}(\mathbf{x}) \ge 0, \ i \neq j \right\} \cap C_{S(x, t)} =  \\
 \quad \quad = \left\{ \mathbf{x} \in \mathbb{R}^{2N} \ : \ D_{ij}(\mathbf{x}) \ge 0, \ i \neq j \right\} \cap \left( \bigcap_{i=1}^{m} \left( \bigcup_{j=1}^{k} \, C_{P_{ij}(x, t)} \right) \right) \, ,
\end{array}
\end{equation}
where $C_{P_{ij}(x, t)}$ is of type (\ref{admiss}) and therefore is definable. Since definable sets are a closed class regarding the finite intersections and unions, $C(t)$ is definable.

As in the case with lack of obstacles, the set of the admissible velocities is the Bouligand tangent cone $T_{C(t)}(\mathbf{x}(t))$ to $C(t)$ at $\mathbf{x}(t)$ for every $t \ge 0.$

Although at first glance in order to obtain the actual velocities it seems natural to project the spontaneous velocities $d(\mathbf{x}(t))$ on the Bouligand tangent cone to $C(t)$ at $\mathbf{x}(t) \, ,$ this is not the correct approach. In order to obtain the actual velocities, we need to consider the first $n$ coordinates of the projection of $d_1(\mathbf{x}(t)) := (d(\mathbf{x}(t)), 1)$ on the Bouligand tangent cone to $K$ at $(\mathbf{x}(t), t) \, ,$ intersected with $\{t=1\} \, .$ The intuitive explanation is that in order to obtain the actual velocity at the moment $t$, we should project our spontaneous velocity on where we expect the obstacles to be at this moment, not on where they are at the current moment. It is backed up by the following example:

\begin{example} \label{ex}
Let us consider the sweeping process with stationary set $C_0$
\begin{equation} \label{stat}
\begin{array}{l}
    \dot \chi(t) \in d -N_{C_0} (\chi(t))\\
    \chi(0) = \chi_0 \in C_0\\
    \chi(t) \in C_0 \ ,
\end{array}
\end{equation}
and the one with moving set $C(t) := C_0 - t.d$ (this is a translation)
\begin{equation} \label{transl}
\begin{array}{l}
    \dot x(t) \in -N_{C(t)} (x(t))\\
    x(0) = x_0 \in C(0)\\
    x(t) \in C(t) \ ,
\end{array}
\end{equation}
where $C_0$ is closed, $d$ is a constant drift term and $t \in [0, T]$. Then $\chi(t)$ is a solution to (\ref{stat}) if and only if $x(t) = \chi(t)-t.d$ is a solution to (\ref{transl}), because $N_{C_0}(\chi(t)) \equiv N_{C(t)}(x(t)) \, ,$ and $\dot x(t) = \dot \chi(t) - d \in -N_{C_0} (\chi(t)) \, .$ 

As in the proof of Theorem \ref{main}, let us denote
$$ K:=\{(x,t)\in \mathbb{R}^{n+1}: \ t \in [0,T] , \ x\in C(t)\} \, ,$$
and $ y(t) = (x(t), t) \, .$ The following connection between the Bouligand tangent cones to $K$ and $C(t_0)$ can be easily obtained
\begin{equation} \label{cones}
T_K(x_0, t_0) = \{(w, \eta) \in \mathbb{R}^{n+1} \, : \, w \in T_{C_0}(x_0) - \eta . d \} \, .
\end{equation}

Let now $\chi(t)$ be a solution to
\begin{equation} \label{g_pr}
\dot \chi(t)\in G(\chi), \ \  \chi(0)=\chi_0\in C_0, \ \chi(t)\in C_0 \, ,
\end{equation} 
where the graph of the multivalued mapping $G$ is the closure of the graph of the mapping $\chi \mapsto Proj_{T_{C_0}(\chi)}d$ (this is the stationary projection process (\ref{projection_pr_stationary}) with single-valued constant perturbation $d$). The upper-semicontinuity of $N_{C_0}$ and Proposition 6.7 on p.219 from \cite{RockWets} (see also Lemma 3.8 from \cite{GeorgRibarska}) yield that $\chi(t)$ is a solution to (\ref{stat}).
It is straightforward to check that the respective solution $x(t) = \chi(t)-t.d$ to (\ref{transl}) satisfies also
\begin{equation} \label{gr1}
 (\dot x(t), 1) \in Pr G_1((\mathbf{0}, 1), x(t), t) \ \ (x(t_0), 0)= (x_0, 0) \in C(0) \times \{0\} \ \ (x(t), t) \in K \, ,
\end{equation}
where the graph of the multivalued mapping $G_1$ is the closure of the graph of the mapping $(d, x, t) \mapsto Proj_{T_K(y) \cap \{t=1\}} (\mathbf{0}, 1)$ (this is the projection process (\ref{projection_process}) with single-valued constant perturbation $(\mathbf{0}, 1)$) 

Indeed, let $t_0 \in [0, T] \, ,  \ w'_0 :=\dot \chi (t_0) \in G(\chi(t_0)) \, .$ Hence, there exist a sequence $\chi_m$ tending to $\chi(t_0)$ and a sequence $w'_m \in Proj_{T_{C_0}(\chi_m)}d \, ,$ such that $ w'_m \longrightarrow_{m\to \infty} w'_0 \, .$ From (\ref{cones}) it follows that $(w_m, 1) := (w'_m - d, 1) \in T_K(y_m) \cap \{t=1\}$ for all $m$. Let $(w, 1) \in T_K(y_m) \cap \{t=1\}$ for some $m$. Then
$$ \| (w_m, 1) - (\mathbf{0}, 1)\| = \|w_m\| = \|w'_m - d\| \le \|w + d - d \| = \| (w, 1) - (\mathbf{0}, 1) \| \, ,$$
because $w'_m \in Proj_{T_{C_0}(\chi_m)}d$ and $w + d\in T_{C_0}(\chi_m)$. 

We have obtained that $(w_m, 1) \in Proj_{T_K(y_m) \cap \{t=1\}} T_K(y_m) \cap \{t=1\}$ and therefore $(\dot x(t_0), 1) = (\dot \chi(t_0)-d, 1) = (w'_0 - d, 1) \in G_1(y_0) \, .$  This means that the solution $x(t) = \chi(t)-t.d$ to (\ref{transl}) is a solution also to (\ref{gr1}).

Thus (\ref{gr1}) is the natural and only counterpart of the stationary projection process (\ref{g_pr}) if the drift term is single-valued and constant.
\end{example}

Let us go back to the crowd motion model. The vector of the actual velocities $\mathbf{\dot x}(t)$ is obtained as follows:
\begin{equation} \label{model2}
\begin{array}{l}
    \mathbf{\dot x}(t) \in Pr \, Proj_{T_{K}(\mathbf{x}(t))\cap \{t=1\}}(d(\mathbf{x}(t)), 1) \\
    \mathbf{x}(0) = \mathbf{x}_0 \in C(0) \\
    \mathbf{x}(t) \in C(t) \ .
\end{array}
\end{equation}
where $T$ is large enough and
$$ K:=\{(\mathbf{x},t)\in \mathbb{R}^{2n+1}: \ t \in [0,T] , \ \mathbf{x}\in C(t)\} \, .$$
Since $K$ may not be prox-regular, the projection may not be unique and even in very elementary cases the mapping $(d, \mathbf{x}, t) \mapsto Pr \, Proj_{T_{K}(\mathbf{x}, t))\cap \{t=1\}}(d, 1)$ is not upper semi-continuous. Hence, it is natural to look for a mapping (as narrow as possible) which is upper semi-continuous and whose images contain $Pr \, Proj_{T_{K}(\mathbf{x}, t))\cap \{t=1\}}(d(\mathbf{x}), 1)$.

Thus, we are going to look for solutions to
\begin{equation}
\begin{array}{l}
    \mathbf{\dot x}(t) \in Pr \, G(d(\mathbf{x}), \mathbf{x}(t), t) \\
    \mathbf{x}(0) = \mathbf{x}_0 \in C(0) \\
    \mathbf{x}(t) \in C(t) \ ,
\end{array}
\end{equation}
where the graph of the multivalued mapping $G$ is the closure of the graph of the mapping $(d, \mathbf{x}, t) \mapsto Proj_{T_K(\mathbf{x}, t) \cap \{t=1\}}(d, 1) \, .$

Since the spontaneous velocities are continuous, Theorem \ref{thProj} gives the desired solution, provided the obstacles are defined by definable sets and are moving in a definable way.

\end{document}